\documentclass[reqno]{amsart}
\usepackage{graphicx,caption,subcaption}
\usepackage{geometry} 
\theoremstyle{theorem}

\newtheorem{lemma}{Lemma} 
 
\theoremstyle{definition}

 \newtheorem{thm}{Theorem}[section]

 \theoremstyle{definition}
 
 \theoremstyle{remark}
 \newtheorem{rem}[thm]{Remark}
 \numberwithin{equation}{section}

\begin{document}

\title[Urysohn's Lemma and the Tietze Extension Theorem]
 {Proofs of Urysohn's Lemma and the Tietze Extension Theorem via the Cantor function}

\author[]{Florica C. C\^{\i}rstea}

\address{
School of Mathematics and Statistics\\
The University of Sydney\\
NSW 2006\\
Australia}

\email{florica.cirstea@sydney.edu.au}



\date{\today}

\begin{abstract}
Urysohn's Lemma is a crucial property of normal spaces that deals with separation of closed sets 
by continuous functions. It is also a fundamental ingredient in proving the Tietze Extension Theorem, another property of normal spaces that deals with the existence of extensions of continuous functions. Using the Cantor function, we give alternative proofs for Urysohn's Lemma and the Tietze Extension Theorem. 
\end{abstract}

\maketitle

\section{Introduction}

Urysohn's Lemma provides the means for proving big theorems in topology such as Urysohn's metrization theorem 
(see Urysohn's final\footnote{At the age of only 26, he drowned while swimming in the ocean.} paper \cite{urysohn25}) and the Tietze Extension Theorem proved by Tietze \cite{tiet} for metric spaces and generalised by Urysohn \cite{uryso} to normal spaces. For further details, see \cite{fol}. Using the Cantor function, we   
give new proofs for Urysohn's Lemma (in Section~\ref{sec11}) and the Tietze Extension Theorem (in Section~\ref{sec3}).
Urysohn's Lemma, originating in the third appendix to Urysohn's paper \cite{uryso}, gives a property that  characterises normal spaces\footnote{A topological space $X$ is {\em normal} if every disjoint closed subsets of $X$ can be included in disjoint open sets.}. 

\begin{thm}[Urysohn's Lemma] \label{th1} If $A$ and $B$ are disjoint closed subsets of a 
	normal space $X$, then there exists a continuous function $f: X \to [0,1]$ such that $f=0$ on $A$ and $f=1$ on $B$.
\end{thm}

Munkres, the author of the popular book \cite{mun}, 
regards the Urysohn Lemma as ``the first deep theorem of the book" 
 (see p. 207 in \cite{mun}). He adds that ``it would take considerably more originality than most of us possess to prove this lemma unless we were given copious hints."  
For the standard proof of Urysohn's Lemma, see \cite[p.~115]{kelley}, \cite[p. 207]{mun} or \cite[p. 102]{will}.  

A function as in Theorem~\ref{th1} is called a {\em Urysohn function}. 
Its existence is crucial to any  of the many approaches to the Tietze Extension Theorem (\cite{blair,grab,ma,ossa,sco}). But, surprisingly, apart from the classical one, it seems that no other constructions of a Urysohn function are known.

We reduce the proof of Theorem~\ref{th1} to the non-trivial case of a connected normal space when 
we construct a new 
Urysohn function (denoted by $F$). This continuous function will take all the values in $[0,1]$ and is obtained by applying the Cantor Staircase function to a non-continuous function from $X$ into the Cantor set. 
Our argument neither relies on nor reduces to the standard proof (see Remark~\ref{rem1}). In turn, if 
$X$ is {\em not} connected, then a Urysohn function for $A$ and $B$ need not take all the values in $[0,1]$. Indeed, if $X=A\cup B$ with $A$ and $B$ distinct connected components of $X$, then  
the only Urysohn function $F$ for $A$ and $B$ is $F=0$ on $A$ and $F=1$ on $B$. Hence, the proof of Theorem~\ref{th1} is simple if $A$ and $B$ belong to different connected components of $X$. The challenging case is when at least one connected component of $X$ has nonempty intersections with both $A$ and $B$. It is precisely this case that we treat with a new method. The general case follows from this one (see Section~\ref{sec11} for details).

We next introduce the Cantor set and Cantor function. 
Algebraically, 
a point $p\in [0,1]$ is in the Cantor set $\mathcal C$ if and only if $p$ has a ternary expansion 
that does not use digit $1$. We thus write $p=0.p_1p_2\ldots p_k\ldots_3$, where 
$ p_k\in \{0,2\}$ for every $k\geq 1$. 
(The subscript $3$ indicates that the expansion is in base $3$.) 
Geometrically, we construct the Cantor set as follows.    
Starting with the interval $[0,1]$, we remove its open middle third interval $(1/3,2/3)$. We apply the same process to the remaining intervals $[0,1/3]$ and $[2/3,1]$. The Cantor set $\mathcal C$ is the set of points in $[0,1]$
that remain after continuing this removal process {\em ad infinitum}. 
At each stage $n\geq 1$ in this construction,  
we remove $2^{n-1}$ open intervals $(\alpha^{(n)},\beta^{(n)})$, where 
\begin{equation} \label{alb}  \alpha^{(n)}=0.\alpha_1^{(n)}\ldots \alpha_{n-1}^{(n)}1_3,\qquad
\beta^{(n)}=0.\alpha_1^{(n)}\ldots \alpha_{n-1}^{(n)}2_{3}.
\end{equation} Here, 
$\alpha_k^{(n)}\in \{0,2\}$ for each $1\leq k\leq n-1$. 
Let $\mathcal L=\cup_{n=1}^\infty \mathcal L_n$, where  
$\mathcal L_n$ 
is the collection of all points $\alpha^{(n)}$ with a ternary representation as in \eqref{alb}. Hence, $\mathcal L_1=\{1/3\}$ and $\mathcal L_2=\{1/9,7/9 \}$. Now, $\mathcal L$ is countable and dense in $\mathcal C$ (but not dense in $[0,1]$). 
All the numbers in $\mathcal L$ are ``endpoints" and right limit points of $\mathcal C$. 
(Each number in $\mathcal L$ has a ternary expansion consisting entirely of $0$s and $2$s. For example, $1 = 0.222\ldots_3$, 
$1/3 = 0.0222\ldots_3$ and $7/9 = 0.20222\ldots_3$.)
 
The Cantor function $\Phi:[0,1]\to [0,1]$  
is continuous, non-decreasing and surjective. It is given by 
$\Phi= \sum_{k=1}^{n-1} \alpha_k^{(n)} 2^{-k-1}+2^{-n}$  
on $  \left[\alpha^{(n)},\beta^{(n)}\right]$ and $\Phi(p)=\sum_{k=1}^\infty p_k 2^{-(k+1)}$ for every $p\in \mathcal C$, where
$p=0.p_1p_2\ldots p_k\ldots_3$ with $ p_k\in \{0,2\}$ for every $k\geq 1$. The binary expansion of any $y\in [0,1]$ 
can be translated into a ternary representation of a number $p\in \mathcal C$ 
by replacing all the $1$s by $2$s. 
Hence, $\Phi(\mathcal C)=[0,1]$ and $\Phi(\mathcal L)=\mathcal D\setminus\{0,1\}$, where  
$\mathcal D$ is the set of all dyadic rationals in $[0,1]$.

\noindent {\bf Notation.} Fix $\alpha^{(n)}\in \mathcal L_{n}$ and $\beta^{(n)}=\alpha^{(n)}+3^{-n}$. 
For $k\geq 1$, let $ q^{(n)}_k:=\alpha^{(n)}-2\cdot 3^{-n-k}$ and 
$ \ell^{(n)}_k:=\beta^{(n)}+3^{-n-k}$. Then, 
$\{q^{(n)}_k\}_{k}$ and $\{\ell^{(n)}_k\}_{k}$ are strictly monotone sequences in $ \mathcal L_{n+k}$ converging to $\alpha^{(n)}$ and 
$\beta^{(n)}$, respectively, satisfying $\Phi(q^{(n)}_k)\nearrow \Phi(\alpha^{(n)})$ and 
$\Phi(\ell^{(n)}_k)\searrow \Phi(\alpha^{(n)})$ as $k\to \infty$. 
If 
 $ p_*=\max\,\{p\in \cup_{j=1}^n \mathcal L_j\cup\{0\}:\   p<\alpha^{(n)} \}$ and  
$ p^{*}=\min\,\{p\in \cup_{j=1}^n \mathcal L_j\cup\{1\}:\   p>\beta^{(n)}  \}$, then   
$p_*,\alpha^{(n)}$ and $p^{*}$ are {\em consecutive} points in $\cup_{j=1}^{n} \mathcal L_{j}\cup\{0,1\}$.

\section{Proof of Urysohn's Lemma} \label{sec11}

Let $A$ and $B$ be disjoint closed subsets of a normal space $X$. 
The proof reduces to the case of a connected space $X$. Indeed, 
the connected components of $X$ form a partition of  $X$: they are disjoint, nonempty, and their union is the whole space. Hence, $X=\cup_{x\in X} \mathcal C(x)$, where $\mathcal C(x)$ is the connected component of $x$. If $x,y\in X$ with $x\not=y$, then either $\mathcal C(x)=\mathcal C(y)$ or $\mathcal C(x)\cap \mathcal C(y)=\emptyset$. Since every connected component is closed and normality is inherited by the closed subspaces of $X$, it follows that the subspace $\mathcal C(x)$ is also normal for every $x\in X$. 
Define $A\cap \mathcal C(x)=A_x$ and $B\cap \mathcal C(x)=B_x$ for each $x\in X$. Then, $A_x$ and $B_x$ are closed and disjoint for every $x\in X$ and, moreover, 
$A=\cup_{x\in X} A_x$ and $B=\cup_{x\in X} B_x$. 
It is enough to construct, for every $x\in X$, a continuous function $F_x:\mathcal C(x)\to [0,1]$ such that $F_x=0$ on $A_x$ and 
$F_x=1$ on $B_x$. (The non-trivial case is when $A_x$ and $B_x$ are nonempty.)   
Then, a Urysohn function for $A$ and $B$ is $F:X\to [0,1]$ given by $F=F_x$ on $\mathcal C(x)$ for each $x\in X$. 

So, if $\mathcal C(x)$, $A_x$, $B_x$ and $F_x$ are relabelled $X$, $A$, $B$ and $F$, respectively, we need only  prove Theorem~\ref{th1} for disjoint closed subsets $A$ and $B$ of a {\em connected} normal space $X$. 

 \vspace{0.1cm}
 \noindent {\bf Step 1.} Set $U_0:=A$ and $U_1:=X\setminus B$. We inductively generate a family $\{U_p\}_{p\in \mathcal L}$ of open neighbourhoods of $A$ such that $\overline{U_p}\subset U_q$ for all $p,q\in \cup_{n\ge 1}\mathcal L_n\cup \{1\}$ with $p<q$. 
 
 For $n=1$, by the normality of $X$, 
the set $U_1$ contains the closure of an open neighbourhood $U_{1/3}$ of 
$A$. Since $X$ is connected, the only nonempty open and closed subset of $X$ is $X$. Thus, we have {\em strict} inclusions $A\subset U_{1/3}$ and $\overline{U_{1/3}}\subset U_1$.  
This verifies Step~1 for $n=1$. 

Fix $n\geq 1$. Assume that $\{U_p\}_{p\in \cup_{j=1}^{n} \mathcal L_j}$ is a family of open neighbourhoods of $A$ satisfying
\begin{equation} \label{two} \tag{$\frak{B}_n$}
\overline{U_p}\subset U_q  \quad \mbox{for all }
p,q\in \cup_{j=1}^{n} {\mathcal L}_j\cup\{1\}  \mbox{ with } p<q. 
\end{equation}   
Let $\alpha^{(n)}\in  \mathcal L_{n}$ be arbitrary. Then,
$q_1^{(n)}\in (p_*,\alpha^{(n)})$ and $\ell_1^{(n)}\in (\alpha^{(n)},p^*)$ are consecutive points in $\mathcal L_{n+1}$. 
By the induction assumption,  
$U_{p^*}$ and $U_{\alpha^{(n)}}$  are open neighbourhoods of $\overline{U_{\alpha^{(n)}}}$ and $\overline{U_{p_*}}$, respectively.  
Thus, $U_{p^*}$ contains the closure of an open neighbourhood 
$U_{\ell_1^{(n)}}$ of $\overline{U_{\alpha^{(n)}}}$, whereas   
$U_{\alpha^{(n)}}$ contains the closure of an open neighbourhood $U_{q_1^{(n)}}$ of $\overline{U_{p_*}}$.    
The collection of all $U_{q_1^{(n)}}$ and $U_{\ell_1^{(n)}}$ obtained by 
varying $\alpha^{(n)}\in \mathcal L_n$ yields the 
family $\{U_q \}_{ q\in  {\mathcal L}_{n+1} }$ of open sets satisfying $(\mathfrak B_{n+1})$.  

\vspace{0.2cm}
\noindent {\bf Step 2.} We define $g=1$ on $X\setminus U_1$, $g=0$ on $A$ and $g(x)=\inf\,\{p\in \mathcal L:\ x\in U_p\}$ for every $x\in U_1\setminus A$. If
$g(x)>p$ for $p\in \mathcal L$, then $x\not\in \overline{U_p}$. Otherwise,
$x\in U_q$ for every $q\in \mathcal L$ with $q>p$. Then, 
$g(x)\leq q$. By letting $q\in \mathcal L$ with $q\searrow p$, we arrive at $g(x)\leq p$, which is a contradiction.

Let $F=\Phi\circ g$ on $X$. 
Then, $F=0$ on $A$ and $F=1$ on $ B$. We prove that $F:X\to [0,1]$ is continuous. 
For any $\zeta\in \mathcal D\setminus\{0,1\}$, there exist $n\geq 1$ and 
$\alpha^{(n)}\in \mathcal L_n$ such that $\zeta=\Phi(\alpha^{(n)})$.  

We have $F^{-1}([0,\zeta))=\cup_{\xi\in \mathcal L\cap (0,\alpha^{(n)})} U_{\xi}$. Indeed, if $x\in U_\xi$ for $\xi\in \mathcal L\cap (0,\alpha^{(n)})$, then $g(x)\leq \xi$, which gives that $F(x)=\Phi(g(x))\leq \Phi(\xi)<\Phi(\alpha^{(n)})=\zeta$. Conversely, 
if $x\in F^{-1}([0,\zeta))$, then  
$F(x)<\Phi(q_k^{(n)})<\zeta$ by taking $k\geq 1$ large enough. Hence, $g(x)<q_k^{(n)}$ so that $x\in U_{q_k^{(n)}}$. 

Similarly, we see that 
$F^{-1}((\zeta,1])=\cup_{ \eta\in 
\mathcal L\cap (\beta^{(n)},1)} (X\setminus \overline{U_{\eta}} )$.  
Indeed, let $x\in X\setminus \overline{U_{\eta}}$ for $\eta\in \mathcal L$ with 
$\eta> \beta^{(n)}$. Then, $g(x)\geq \eta$ and, hence, $F(x)\geq \Phi(\eta)>\zeta$. 
Conversely, if $x\in F^{-1}((\zeta,1])$, then $F(x)>\Phi(\ell^{(n)}_k) > \zeta$ for $k\geq 1$ large enough. We have 
$g(x)> \ell^{(n)}_k$ and, hence, $x\not \in \overline{U_{\ell^{(n)}_k}}$.  

As $\mathcal S=\{[0,\tau),\ (\tau,1]:\ 0<\tau<1 \}$ is a subbase for $[0,1]$ and $\mathcal D$ is dense in $[0,1]$, the continuity of 
$F$ follows using that $F^{-1}([0,\zeta))$ and $F^{-1}((\zeta,1])$ are open for any dyadic 
rational $\zeta$ in $(0,1)$. \qed

\begin{rem} \label{rem1} The standard approach of Urysohn's Lemma comprises three steps: 
(i) construction of a family $\{U_r\}_{r\in \mathcal D}$ of open sets indexed by\footnote{The index set $\mathcal D$ can be any subset of $\mathbb Q$ that is dense in $[0,1]$.} the dyadic rationals $r=j/2^k$ in the interval $[0,1]$ such that $A\subseteq U_0$, $B=X\setminus U_1$ and 
$\overline{U_r}\subseteq U_s$ whenever $r<s$;   
(ii) verification by induction that the family of open sets $\{U_r\}_{r\in \mathcal D}$ has the required properties;  
(iii) construction of the continuous function, namely,   
$f(x)=\inf \{ r:\ x\in U_r\}$ for $x\in X\setminus B$ and $f=1$ on  $ B$. 

In contrast, we reduce the proof to the non-trivial case of a connected normal space $X$ and 
construct our Urysohn function $F:X\to [0,1]$ to be surjective.  
The connectedness of $X$ is used in Step~1 to 
obtain {\em strict} inclusions $\overline {U_p}\subset U_q$ for all $p,q\in \mathcal L\cup \{1\}$ with $p<q$. 
For our family $\{U_p\}_{p\in \mathcal L}$ of open sets, the index set $\mathcal L$ is {\em not} dense in $[0,1]$. It is, however, dense in the Cantor set. 
The continuity of $F(=\Phi\circ g)$ follows essentially as a result of composing the Cantor function with $g$ in Step~2. And 
$g$ is {\em never} continuous on the connected space $X$ since $g$ takes values in the Cantor set (a perfect set that is nowhere dense). 
\end{rem}

\section{Proof of the Tietze Extension Theorem} \label{sec3}

Using our new Urysohn function,  
we give an alternative proof of 
the Tietze Extension Theorem (see Theorem~\ref{tiez}). We use  
the following result, which is easy to establish (see \cite[Lemma~1]{ossa}). 
	
	\begin{lemma}  \label{cons} Let $E$ and $Y$ be closed subspaces in a normal space $X$ and let $U$ be an open neighbourhood of $Y$ in $X$. Let a subset $C$ of $E$ be a closed neighbourhood in $E$ of $Y\cap E$ such that $C\subseteq U\cap E$. Then, $Y$ admits a closed neighbourhood $Z$ that is included in $U$ and $Z\cap E=C$.  
	\end{lemma}

\begin{thm} 
\label{tiez} 
	Let $E$ be a closed subspace of a normal space $X$. Then, every continuous function $f:E\to [0,1]$ can be extended to a continuous function $F:X\to [0,1]$. 
\end{thm}

\begin{proof} As for Urysohn's Lemma, 
we consider each connected component $\mathcal C(x)$ of $X$ and use that $E\cap \mathcal C(x)$ is a closed subset of the normal subspace $\mathcal C(x)$. Hence, it is enough to prove Theorem~\ref{tiez} when $X$ is a connected normal space.

\noindent {\bf Case I.} Let $f:E\to [0,1]$ be a continuous and {\em surjective} function. 
The sets $A=f^{-1}(0)$ and $B=f^{-1}(1)$
	are disjoint and closed in $E$ (and, hence, in $X$). 
Define $U_0=A$ and $U_1=X\setminus B$.
 
	For $Z \subseteq X$, we set $Z^c:=X\setminus Z$. We construct 
 open neighbourhoods $\{U_p\}_{p\in \mathcal L}$ of $A$ as in Step~1 of Urysohn's Lemma and, in addition,    
	$U_p\cap E= f^{-1}([0,\Phi(p)))$ 
		for every $p\in \mathcal L$. 
		More precisely, for each $n\geq 1$, we generate open neighbourhoods $\{U_q\}_{q\in \mathcal L_n}$ of $A$ satisfying 
		\eqref{two} and
		 \begin{equation} \label{ops} \tag{$\frak{D}_n$} U_q^c\cap E= f^{-1}([\Phi(q),1])\quad \mbox{for every }q\in  \mathcal L_n.
	\end{equation}
By Lemma~\ref{cons}, 
$B$ has a closed neighbourhood $U_{1/3}^c$ contained in $A^c$ with  $U_{1/3}^c\cap E=f^{-1}([1/2,1])$. 
This proves the claim for $n=1$.  		
For $n\geq 1$, assume that  
	$\{U_p\}_{p\in \cup_{j=1}^{n} \mathcal L_j}$ is a family of open neighbourhoods of $A$ satisfying 
\eqref{two} and \eqref{ops}. 
For fixed $\alpha^{(n)}\in \mathcal L_{n}$, 	
let $p_*,p^*, q^{(n)}_1$ and $\ell^{(n)}_1$ be as in \S\ref{sec11}. 
Using the induction assumption and Lemma~\ref{cons}, we find that $ U_{p^*}^c$ has a   
closed neighbourhood $U_{\ell_1^{(n)}}^c$ contained in 
 $(\overline{ U_{\alpha^{(n)}}})^c $ and   
$(\mathfrak D_{n+1})$ holds for $q=\ell_1^{(n)}\in \mathcal L_{n+1}$. 
Similarly, $U_{\alpha^{(n)}}^c$ has a 
closed neighbourhood $U_{q_1^{(n)}}^c$ contained in   
$ (\overline{ U_{p_*}})^c $ and $(\mathfrak D_{n+1})$ holds for $q=q_1^{(n)}\in \mathcal L_{n+1}$. 
All $U_{q_1^{(n)}}$ and $U_{\ell_1^{(n)}}$ obtained by varying $\alpha^{(n)}\in \mathcal L_n$ 
yield the family $\{U_q \}_{ q\in  {\mathcal L}_{n+1} }$ satisfying $(\mathfrak B_{n+1})$ and $(\mathfrak D_{n+1})$. 

Let $F:X\to [0,1]$ be our Urysohn function associated to $\{U_p\}_{p\in \mathcal L}$.
For any $\zeta \in \mathcal D\setminus\{0,1\}$, there exist $n\geq 1$ and $\alpha^{(n)}\in \mathcal L_n$ with 
$\zeta=\Phi(\alpha^{(n)})=\Phi(\beta^{(n)})$. By the density of $\mathcal L$ in $\mathcal C$, for every $\rho\in \mathcal L$ with $\rho>\beta^{(n)}$, there exists $\eta\in \mathcal L\cap (\beta^{(n)},\rho)$, which yields that $U_\rho^c\subseteq X\setminus \overline{U_{\eta}}$.  
Then, by Step~2 in \S\ref{sec11}, 
$ F^{-1}((\zeta,1])=\cup \{ U_{\rho}^c:\ \rho\in \mathcal L,\ \rho>\beta^{(n)} \}$ and $F^{-1}([0,\zeta))=\cup \{U_{\xi}:\ \xi\in \mathcal L,\ \xi<\alpha^{(n)} \}$. Thus, $E\cap F^{-1}((\zeta,1])=f^{-1}((\zeta,1])$ and 
$ E\cap F^{-1}([0,\zeta))=f^{-1}([0,\zeta))$.  
These equalities extend to every $\zeta\in (0,1)$ by density of $\mathcal D$ in $[0,1]$. 
Hence, $F:X\to [0,1]$ is a continuous extension of $f$.    

\vspace{0.2cm}
\noindent {\bf Case II.} Let $h:E\to [0,1]$ be any continuous function, where $E\subset X$ is closed. 
If $X$ were to be the only open neighbourhood of $E$, then necessarily,
$\overline{\{x\}}\cap E\not=\emptyset$ 
for every $x\in X\setminus E$. 
Any continuous extension of $h:E\to [0,1]$ to the whole $X$, say $F$, would need to  
be sequentially 
continuous, yielding the following definition:   
for every $x\in X\setminus E$, if $y\in \overline{\{x\}}\cap E$, then  
$F(x)=h(y)$.  

Thus, without loss of generality, we assume that there exists an open neighbourhood $V_2$ of $E$ such that $V_2\subset X$. 
By the normality and connectedness of $X$, there exists an open set $V_1$ such that $E\subset V_1$ and $\overline{V_1}\subset V_2$ (with strict inclusions).  
Urysohn's Lemma gives a continuous function $\varphi:X\to [0,1]$
with 
$\varphi=0$ on $\overline{V_1}$ and $\varphi=1$ on $V_2^c$.
We have $\varphi(\overline{V_2}\setminus V_1)=[0,1]$ by the connectedness of $X$.  
If $f=\varphi$ on $\overline{V_2}\setminus V_1$ and $f=h$ on $E$, then $f:(\overline{V_2}\setminus V_1) \cup E\to [0,1]$ is continuous 
and surjective.   
By Case I, $f$ (and thus $h$) has a continuous extension $F:X\to [0,1]$. 
	\end{proof}

\section{Notes on the Cantor set and Cantor function} \label{sec2}

The Cantor set\footnote{Fleron \cite{fle} noted that Cantor was not the first to uncover general ``Cantor sets."  
	Such sets featured earlier in a paper \cite{smi} of Smith, who discovered and constructed nowhere dense sets with outer content.}  and 
Cantor function are two of Cantor's ingenious creations that go back to 1883. During the years 1879--1884, G. Cantor (1845--1918) gave the first systematic treatment of the point set topology of the real line in a series of papers entitled ``\"Uber unendliche, lineare Punktmannichfaltigkeiten." Among the terms he introduced and still in current use, we mention two: {\em everywhere dense} set and {\em perfect set}. The terminology (but not the concept\footnote{The concept of limit point was conceived by Weierstrass who, without giving it a name,  
	used it between 1865 and 1886 as part of his statement that {\em any infinite bounded set in $n$-dimensional Euclidean space has a limit point} (known as the Bolzano--Weierstrass Theorem).}) 
of ``limit point", along with the notion of derived set, was introduced by Cantor in a paper of 1872. 
 
The Cantor set ranks as one of the most frequently quoted fractal objects in the literature. It emerges again and again in many areas of mathematics from topology, analysis and abstract algebra to fractal geometry \cite{man,va}.  
The Cantor set appeared in a footnote to Cantor's statement \cite{can} that perfect sets need not be everywhere dense. Without any indication on how he came upon it, Cantor noted that this set is an infinite and perfect set that is {\em nowhere dense}\footnote{A set is called nowhere dense if the interior of its closure is empty.} in any interval, regardless of how small it is. 
The first occurrence of the Cantor function is in a letter by Cantor \cite{can2} dated November 1883. 
The Cantor function in \cite{can2} served as a counterexample to Harnack's extension of the Fundamental Theorem of Calculus to discontinuous functions. 

The properties of the Cantor function (also called the Lebesgue function or the Devil's Staircase) are surveyed in \cite{dov}.  
For the history of 
the Cantor set and Cantor function, see 
\cite{fle}.

\end{document}